\documentclass[letterpaper, 10 pt, conference]{ieeeconf}  

\IEEEoverridecommandlockouts                              
\usepackage{graphics} 
\usepackage{epsfig} 
\usepackage{mathptmx} 
\usepackage{times} 

\usepackage{amsmath,amssymb,amsthm,mathtools,empheq,dsfont}
\usepackage{xcolor}

\def\balpha{\boldsymbol{\alpha}}

\newcommand\bq{\boldsymbol{q}}
\newcommand\br{\boldsymbol{r}}
\newcommand\bxbar{\bar{\boldsymbol{x}}}

\usepackage{algorithmic}
\usepackage[linesnumbered,ruled,vlined]{algorithm2e}

\newcommand\EE{\mathbb E}

\newcommand\RR{\mathbb R}

\DeclareMathOperator*{\argmin}{arg\,min}
\def\E{\mathbb{E}}
\def\R{\mathbb{R}}

\theoremstyle{plain}

\newtheorem{proposition}{Proposition}

\newtheorem{definition}{Definition}
\newtheorem{theorem}{Theorem}

\theoremstyle{remark}
\newtheorem{remark}{Remark}[section]
\def\mff{m}          
\def\raf{\mu}        

\newcommand{\revision}[1]{{\color{blue}#1}}

\title{\LARGE \bf
Inverse Learning of the Altruism and Cost Level in\\ Mixed-Individual Mean Field Games
}

\author{Haoyang Cao, Gökçe Dayanıklı, and Xiaofei Shi
\thanks{H. Cao is with Department of Applied Mathematics and Statistics,
        Johns Hopkins University
        {\tt\small hycao@jhu.edu}}%
\thanks{G. Dayanıklı is with Department of Statistics,
        University of Illinois, Urbana-Champaign,
        {\tt\small gokced@illinois.edu}}%
\thanks{X. Shi is with Department of Statistical Sciences, University of Toronto,
        {\tt\small xf.shi@utoronto.ca}}%
}

\begin{document}

\maketitle
\thispagestyle{empty}
\pagestyle{empty}

\begin{abstract}

Understanding how humans respond to incentives, both at the individual and collective levels, is crucial to the design of effective policies. Within the continuous-time stochastic framework for large interacting populations, mean field games (MFGs) model populations of non-cooperative agents, whereas mean field control (MFC) describes the fully cooperative benchmark, interpreted in our setting as fully altruistic behavior. 
Mixed-individual MFGs interpolate between these two extremes through a parameter governing the degree of altruism. 
A central challenge for regulators and policymakers, however, is that intrinsic altruism levels and other private structural parameters, such as individual labor costs, are typically unobservable. 
To address this challenge, we develop an inverse learning framework for mixed-individual MFGs. Our approach enables the recovery of (latent) altruism and labor cost levels from noisy observations, with experiments demonstrating the feasibility and accuracy of our method. These findings underscore the promise of inverse MFG methodologies for uncovering latent preference structures in large populations, with important implications for incentive design, empirical behavioral modeling, and data-driven policy analysis.

\end{abstract}

\section{Introduction}
Mean field control (MFC)~\cite{bensoussan2013mean,carmona2013control} and mean field games (MFGs)~\cite{huang2006nash,lasryJeuxChampMoyen2006,lasryMeanFieldGames2007} have increasingly become a tool for modeling and analyzing large populations of interacting agents in different applications such as modeling of energy markets~\cite{aid2020entry,carmona2022mean,djehiche2020price}, financial systems~\cite{carmonafouquesun2015mean,lachapelle2016efficiency}, or traffic management~\cite{chevalier2015micro,huang2021dynamic}. These models approximate solutions to finite-agent systems by considering the limit in which the number of agents goes to infinity, focusing on the interactions between a representative agent and the population distribution of the states or controls of the agents. 
Although related, the two frameworks differ in their underlying strategic settings: MFC models cooperative agents and approximates the social optimum, whereas MFG models non-cooperative agents and approximates a Nash equilibrium. The distinction in their solutions is a consequence of how individuals perceive their influence on the mean field term. In MFC, each agent endogenizes their impact on the mean field and optimizes accordingly. In contrast, in MFG, each agent treats the mean field as exogenous, optimizes her best response to it, and the equilibrium is achieved by a fixed point argument.

However, in many real life applications especially the ones including the modeling of the human behavior, agents may display a mixture of non-cooperative and cooperative behaviors. Following this motivation mixed individual MFGs (MI-MFGs) have been introduced~\cite{totc_cdc,dayanikli2025cooperation}. In this model, each agent sees their effect on the mean field term depending on an altruism parameter. Varying this altruism parameter allows the model to smoothly transition from fully non-cooperative behavior to fully cooperative control. This modeling flexibility enables a more realistic representation of agent behavior in socio-economic and engineering systems.

From a mechanism design perspective, improving the societal outcomes requires understanding the underlying preferences and objectives of interacting individual agents. In practice, model parameters or their level of altruism are typically unknown and cannot be directly observed, but they can be inferred from observed equilibrium behavior. While MI-MFGs provide a flexible framework to model the combination of cooperative and non-cooperative behavior, existing works have mainly focused on the equilibrium analysis for a known model and in contrast learning of characteristics in a data-driven way remains unexplored. In this paper, we address this gap by studying the inverse learning of both the cost parameter and the altruism parameter in the MI-MFG. One of the motivations for proposing a data-driven learning method is the fact that they can be easily generalized to the higher dimensions.

Related literature on inverse learning in mean-field models is still sparse and is considerably more developed for MFGs than for MFC. For MFGs, \cite{yang2018learning} studied a class of discrete-time models in which the inverse problem can be reduced to an induced Markov decision process and addressed using deep inverse reinforcement learning. \cite{chen2022individual} later showed that this reduction is essentially limited to cooperative settings and proposed an individual-level inverse-learning formulation that directly captures mean-field interactions; see also \cite{chen2023adversarial} for an adversarial extension designed to improve robustness to imperfect demonstrations. More recently, \cite{anahtarci2025maximum,anahtarci2025kernel} developed maximum-causal-entropy and kernel-based formulations for stationary MFGs, providing more structured optimization frameworks and richer reward classes. By contrast, inverse problems for MFC have received much less attention and have mostly appeared in the forms of inverse optimal control and data-driven identification; representative examples include the inverse traffic-type MFC model in \cite{kachroo2016inverse} and the data-driven inverse approach for potential mean-field models in \cite{zhang2025surrogate}.

Our contributions are three-fold. The first contribution is related to the analysis of MI-MFGs. Different from~\cite{totc_cdc} and~\cite{dayanikli2025cooperation} that use forward-backward stochastic differential equations, we prove characterization of mixed individual mean-field Nash equilibrium (MI-MFNE) via Hamilton-Jacobi-Bellman (HJB) and Kolmogorov-Fokker-Planck (KFP) forward-backward partial differential equations (FBPDEs). We also reduce the FBPDE system to a system of coupled forward-backward ordinary differential equations (FBODEs) and give the existence and unique results for the MI-MFNE. Our second contribution is to introduce the \textit{first inverse learning method} for the MI-MFGs with the motivation of understanding the altruism levels and cost parameters of individuals which may not be known explicitly from the observed behaviors.  Finally, we validate the theoretical inverse learning findings with experiments. In these experiments, we demonstrate how altruism influences equilibrium behavior.

The paper is structured as follows. Section~\ref{sec:model} introduces the mathematical setting and related equilibrium definition for an MI-MFG model of our interest and formalize our problem statement. We emphasize the perspectives of both the regulator and individual agents in this game setting. Section~\ref{sec:method} presents the theoretical analysis of the MI-MFG model and the inverse learning method. Section~\ref{sec:numerics} introduces the inverse learning algorithm to learn the altruism levels and cost parameters and presents the experiment results that validate the theoretical findings. Finally in Section~\ref{sec:conclusion}, we summarize our findings to conclude and give our future directions.

\section{Model}
\label{sec:model}
\subsection{Mixed Individual Mean-Field Game}\label{subsec:II-A}
Inspired by~\cite{dayanikli2025cooperation}, we introduce the formulation of a mixed individual mean-field game (MI-MFG) over a finite time horizon \([0,T]\) with fixed \(T>0\). Given a filtered probability space \((\Omega,\mathcal F,\mathbb P,\mathbb F=(\mathcal F_t)_{t\geq0})\) satisfying the usual conditions that supports a one-dimensional standard Brownian motion \(W=(W_t)_{t\geq0}\), we let \(\xi\in\mathcal F_0\) such that \(\xi\sim \mu_0\) and \(\mathbb E[\xi^2]<\infty\). For a representative agent \(i\), with $b_\alpha$ and $b_x$ being constant parameters and $\sigma>0$ as the constant volatility, her state process evolves according to
\small
\begin{align}\label{eq:dyn}
dX_t^\alpha = \left(b_\alpha \alpha_t + b_x X_t \right) dt + \sigma dW_t,  \qquad X_0 \sim \xi, 
\end{align}
\normalsize
under a control \(\alpha=(\alpha_t)_{t\in[0,T]}\in\mathcal A\). Here, \(\mathcal A\) denotes the set of admissible controls such that
\small
\begin{equation}
    \label{eq:adm-ctrl-1}
    \begin{aligned}
    \mathcal A\coloneqq\biggl\{&(\alpha_t)_{t\in[0,T]}\biggl|\exists[0,T]\times\R\overset{\alpha}{\mapsto}\R\text{ s.t.  }\alpha_t=\alpha(t,X^\alpha_t)\\&\text{and \eqref{eq:dyn} has a unique strong solution with}\\&
    \mathbb E\bigl[\sup_{t\in[0,T]}(X^\alpha_t)^2\bigl]<\infty\biggl\}.
    \end{aligned}
\end{equation}
\normalsize
Given a flow of probability distributions \(\mff=(\mff_t)_{t\in[0,T]}\), Agent \(i\) faces the following expected total cost 
\small
\begin{equation}\label{eq:agent-cost}J(\alpha;\mff) = \frac{1}{2}\E\left[\int_0^T c_\alpha \alpha_t^2 + c_x \left(X_t^\alpha\right)^2 + c_\mu f(\raf_t^\alpha, \mff_t;\lambda) \ dt \right],\end{equation}
\normalsize
subject to \eqref{eq:dyn}, where $c_\alpha, c_x, c_\mu>0$ are constant coefficients. Here, the cost of interaction \(f\) is given by
\small
\begin{equation}\label{def: mean penalty}
f(\raf^{\alpha}, \mff;\lambda)
= \left(\lambda\E_{Z\sim \mff}[Z] + (1-\lambda)\E_{Z\sim \raf^{\alpha}}[Z] \right)^2,
\end{equation}
\normalsize
with \(\mu^\alpha=(\mu^\alpha_t)_{[0,T]}=({\rm Law}(X^\alpha_t))\) and \(\lambda\in[0,1]\). 
Our model is presented under the one-dimensional state, control, and noise case for the sake of simplicity in presentation and notation. The technical analysis and learning approach can be extended to the higher dimensional setting in a straightforward manner.
Define the following notion of equilibrium:
\begin{definition}\label{def:mf nash equilibrium}
A pair $(\hat \alpha, \hat\mff)$ is said to be a mixed individual mean-field Nash equilibrium (MI-MFNE) provided that
\begin{enumerate}
\item Given the distribution flow $\hat \mff$, $\hat \alpha \in \argmin_{\alpha\in\mathcal{A}} J(\alpha; \hat\mff)$ subject to \eqref{eq:dyn}, and in the meantime,
\item $\hat \mff_t = \raf^{\hat\alpha}_t =\mathcal{L}(X_t^{\hat\alpha})$ for all $t\in[0,T]$. 
\end{enumerate}
\end{definition}

From the above definition, we can see that the agents in this MI-MFG are interacting with each other through the cost structure in the minimization problem. In particular, the presence of the cost of interaction \eqref{def: mean penalty} encompasses an intermediate status between a non-cooperative game (i.e., an MFG) and a fully cooperative setting (i.e., an MFC). In this way, intuitively, the altruism parameter \(\lambda\) reflects the level of altruism of the agent: when \(\lambda\) is close to \(1\), an agent is more of a non-cooperative player in the game, competing with her opponents via the mean field \(\mff\) in the cost; when \(\lambda\) is close to \(0\), on the other hand, the agent is more collaborative with the rest of the population through \(\mu^\alpha\). We note that, strictly speaking, $(1-\lambda)\in[0,1]$ represents the \textit{altruism level}; for simplification, we will call $\lambda$ the altruism parameter.

\subsection{Regulator versus Individual Agents}
As suggested by~\cite{dayanikli2025cooperation}, a general MI-MFG can be used to study the phenomenon of ``tragedy of the commons'', such as overfishing, deforestation, traffic congestion, and so on in a more realistic manner. In these models, agents' state processes can be interpreted as individual wealth from exploiting common goods. The mean information models exhaustion of common goods. The individual overall cost functions usually take into account both the cost of labor and the regulatory cost to prevent the tragedy of the commons; agents' level of altruism will then impact the effectiveness of such regulatory costs.

Back to the specific game setting introduced in Section~\ref{subsec:II-A}, we now put it into the perspectives of regulator and individual agents, respectively. We assume that the state process \eqref{eq:dyn} is public to both the regulator and the agents, that is, \(b_x,b_\alpha\in\R\) are given. In the total cost \eqref{eq:agent-cost}, the individual agents determine the coefficient in the quadratic labor cost \(c_\alpha>0\) as well as the level of altruism via the altruism parameter \(\lambda\in[0,1]\); the regulator determines the \(c_x,c_\mu\in(0,+\infty)\) in the regulatory cost. From the regulator's perspective, it is crucial to understand how agents respond to the regulatory cost when the altruism parameter \(\lambda\) and labor cost structure \(c_\alpha\) are private to the agent. For the rest of this paper, we study how to infer \(\lambda\), or jointly infer \((\lambda,c_\alpha)\) by observing agents' equilibrium behavior under a given game setting.  

\section{Method and Results}
\label{sec:method}
\subsection{Analysis of the Mixed Individual MFG}
\label{subsec:mi-mfg}
We first characterize the equilibrium strategy for the representative agent \(i\) under a given game structure. The characterization is given by a coupled forward-backward partial differential equations (FBPDE).

\begin{theorem}[FBPDE characterization of MI-MFNE]
\label{the:FBPDE}
A control $\hat{\balpha}$ is an MI-MFNE control profile if and only if:
\small
\begin{equation}\label{eq:mi-BR-formula}
\hat \alpha_t = -\frac{b_\alpha \partial_x u(t,x)}{c_\alpha},
\end{equation}
\normalsize
where $(u, m)=(u_t, m_t)_{t\in [0,T]}$ solves the following forward-backward partial differential equation (FBPDE) system of Kolmogorov-Fokker-Planck (KFP) and Hamilton-Jacobi-Bellman (HJB) equations:
\small
\begin{equation}\label{eq:MI_FBSDE}
\left\{
\begin{aligned}
-\partial_t u_t &= \tfrac{1}{2} \sigma^2 \partial^2_{xx} u(t,x)  -\tfrac{b_{\alpha}^2}{2c_\alpha} \big(\partial_x u(t,x)\big)^2 + b_x x \partial_x u(t,x) \\
&\hskip5mm +  \tfrac{c_x}{2} x^2+ \tfrac{c_\mu}{2} \bar x_t^2+ c_\mu (1-\lambda) \bar x_t x ,
\\[1mm]
\partial_t m_t &=\tfrac{\sigma^2}{2}\partial^2_{xx}m(t,x) - \partial_x\big(m(t,x)\big(-\tfrac{b_\alpha^2}{c_\alpha}\partial_x u(t,x)+b_x x\big) \big) ,\\[1mm]
\qquad u_T &= 0,\qquad m(0,x) = \mu_0(x),
\end{aligned}
\right.
\end{equation}
\normalsize
where $\bar{x}_t=\int x m(t,x) dx$ is a function of time.
\end{theorem}
\begin{proof}
The first step of the proof is to realize that given the population distribution $\hat m$, the problem becomes a mean field control problem. Following the proof in~\cite{bensoussan2013mean} and by introducing an exogenous population distribution $\hat m$, we first write the Hamiltonian:
\small
\begin{equation*}
\begin{aligned}
H(x,\alpha,\hat m, \mu,y;\lambda) :=& 
\tfrac{c_\mu}{2} \left(\lambda \int \eta \hat m(\eta) d\eta + (1-\lambda) \int \xi \mu(\xi) d\xi \right)^2\\
&+ \left(b_{\alpha} \alpha + b_x x\right)y + \tfrac{c_\alpha}{2} \alpha^2+ \tfrac{c_x}{2} x^2.
\end{aligned}
\end{equation*}   
\normalsize
Due to the convexity, the Hamiltonian is minimized if and only if $\hat\alpha=-\tfrac{b_\alpha y}{c_\alpha}$. The minimized Hamiltonian is therefore:
\small
\begin{equation*}
\begin{aligned}
\hat{H}(x,\hat m, \mu,y;\lambda) :=&  
\tfrac{c_\mu}{2} \left(\lambda \int \eta \hat m(\eta) d\eta + (1-\lambda)\int \xi \mu(\xi) d\xi\right)^2
\\
& -\tfrac{b_{\alpha}^2}{2c_\alpha} y^2 + b_x x y +  \tfrac{c_x}{2} x^2.
\end{aligned}
\end{equation*}
\normalsize
Then the HJB is given by:
\small
\begin{equation*}
\begin{aligned}
-\partial_t u(t,x)  =&\tfrac{1}{2} \sigma^2 \partial^2_{xx} u(t,x)+ \hat{H}(x,\hat m, \mu,\partial_x u(t,x);\lambda) \\
&\hskip5mm+ \int_\RR \dfrac{\partial \hat H}{\partial \mu} (\xi, \hat m, \mu, \partial_x u(t,\xi);\lambda)(x)\mu(\xi) d\xi 
\end{aligned}
\end{equation*}
\normalsize
where we have:
\small
\begin{equation*}
\begin{aligned}
&\dfrac{\partial \hat H}{\partial \mu} (\xi, \hat m, \mu, \partial_x u(t, \xi);\lambda)(x)  \\
&\hskip3mm= c_\mu \Big( (1-\lambda) \lambda \big(\int \eta \hat m(\eta) d\eta\big) x + (1-\lambda)^2  \big(\int \eta \mu(\eta) d\eta \big)x\Big),
\end{aligned}
\end{equation*}
\normalsize
with $u(T,x)=0$ due to zero terminal cost. The coupled forward equation is written by using KFP equation and plugging in the minimizer of the Hamiltonian.
    
The second step is to ensure the consistency condition holds. In other words, at the equilibrium $\hat \mff_t=\raf_t^{\hat\alpha}=\mu_t$. By introducing $\int \eta \hat m d\eta = \int \eta \mu(\eta) d\eta=\bar x$, we can conclude the final system.
\end{proof}

\begin{remark}
It is worth emphasizing that the backward equation in \eqref{eq:MI_FBSDE} is not from dynamic programming principle but rather a variational approach introduced in \cite{bensoussan2013mean}. Therefore, the function \(u(t,x)\) does not refer to the game value in equilibrium. Note that in the definition of admissible control set \(\mathcal A\) in \eqref{eq:adm-ctrl-1}, the feedback control depends only on the state, not the augmented state-law space. Therefore, the \(u\) function corresponds to an adjoint state in the variational approach.  
\end{remark}

We can further specify the analytical solution of the equilibrium strategy.

\begin{proposition}[FBODE characterization of MI-MFNE]
\label{prop:mi_FBODE}
Let $(r_t, q_t, s_t, \bar{x}_t)$ denote the solution of the following forward-backward ordinary differential equation (FBODE) system:
\small
\begin{empheq}[left={\empheqlbrace}]{align}
-\dot{r}_t &= -\frac{2b_\alpha^2}{c_\alpha} r_t^2 + 2b_x r_t +\frac{c_x}{2}, \qquad && r_T = 0,\label{FBODE:riccati}\\[1mm]
-\dot{q}_t &= -\frac{2b_\alpha^2}{c_\alpha} r_t q_t + b_x q_t +c_\mu (1-\lambda) \bar{x}_t, \qquad && q_T = 0,\label{FBODE:q}\\[1mm]
-\dot{s}_t &= \sigma^2 r_t -\frac{b_\alpha^2}{2c_\alpha} q^2_t +\frac{c_\mu}{2}  (\bar{x}_t)^2, \qquad && s_T = 0,\label{FBODE:s}\\[1mm]
\dot{\bar{x}}_t&=  -\frac{b_\alpha^2}{c_\alpha} (2r_t \bar{x}_t +q_t) +b_x\bar{x}_t, &&\bar{x}_0 =\bar{\mu}_0\label{FBODE:xbar},
\end{empheq}
\normalsize
where $\bar{\mu}_0=\int_\RR \xi \mu_0(\xi)d\xi$. Then 
\small
\begin{equation}\label{eq:equilibrium-ctrl-minors-only}
\hat \alpha_t= -\frac{b_\alpha}{c_\alpha} (2r_t X_t + q_t) 
\end{equation}
\normalsize
is the MI-MFNE control. 
    
\end{proposition}
\begin{proof}
Following the linear quadratic form of the model, we propose the following ansatz $u(t,x)= r_t x^2+ q_tx +s_t$ where $r_t, q_t, s_t$ are deterministic functions of time. Plugging in $\partial_t u(t,x) = \dot{r}_t x^2 +\dot{q}_t x + \dot{s}_t$, $\partial_x u(t,x) = 2r_t x +q_t$, and $\partial^2_{xx} u(t,x) = 2r_t$ in the backward differential equation and matching the terms including $x^2$, $x$, and without any $x$, we get the backward ODE system presented in~\eqref{FBODE:riccati}-\eqref{FBODE:s}. The forward ODE is found by plugging in the forward PDE below:
\small
\begin{equation*}
\begin{aligned}
  \dot{\bar{x}}_t &= \frac{d}{dt} \int x m(t,x) dx = \int x \partial_t m(t,x) dx,\\
  &= \int x \tfrac{\sigma^2}{2}\partial^2_{xx}m(t,x) - \nabla_x\big(m(t,x)\big(-\tfrac{b_\alpha^2}{c_\alpha}\partial_x u(t,x)+b_x x\big) \big)dx.
\end{aligned}
\end{equation*}
\normalsize
By plugging in $\partial_x u(t,x) = 2r_t x +q_t$ in the above equation and using integration by parts we conclude:
\small
\begin{equation*}
    \dot{\bar{x}}_t = -\frac{b_\alpha^2}{c_\alpha} (2r_t \bar{x}_t +q_t) +b_x\bar{x}_t.
\end{equation*}
\normalsize
\vskip-5mm
\end{proof}

Finally, we emphasize the uniqueness of such an equilibrium strategy.
\begin{theorem}
    There exists a unique MI-MFNE under small time condition.
\end{theorem}
\begin{proof}
    Following Theorem~\ref{the:FBPDE} and Proposition~\ref{prop:mi_FBODE}, finding MI-MFNE is equivalent to solving the FBODE system in Proposition~\ref{prop:mi_FBODE}. We first emphasize that~\eqref{FBODE:riccati} is a scalar Riccati equation and it has the following explicit unique solution:
    \small
    \begin{equation}\label{eq:riccati-sol}
        r_t = \frac{\frac{c_x}{2} \big(e^{2{R}(T-t)}-1\big)}{({R}-b_x)e^{2{R}(T-t)}+(b_x+{R})},
    \end{equation}
    \normalsize
    where $R = \sqrt{b_x^2 + \tfrac{b_\alpha^2c_x}{c_\alpha}}$.
    Given $\br= (r_t)_{t\in[0,T]}$ process, the backward ODE for $\bq = (q_t)_{t\in[0,T]}$ and the forward ODE for $\bxbar = (\bar{x}_t)_{t\in[0,T]}$ are coupled.

Given any $\bq$, we can write another mapping $\bxbar = \phi_1 (\bq)$ and given any $\bxbar$ we can write a mapping $\bq^\prime = \phi_2(\bxbar)$. In order to show that the coupled system of $\bq$ and $\bxbar$ has a unique solution, we need to show that the composition mapping $\phi = \phi_2 \circ \phi_1: \bq \mapsto \bq^\prime$ is a contraction. 
We define the sup norm for $f\in C([0,T])$ as $||f||_T:= \sup_{t\in[0,T]} |f(t)|$, furthermore we define notation $\Delta g = g^1-g^2$ where $g\in\{q_t, \bar{x}_t, \dot{q}_t, \dot{\bar{x}}_t\}$. 

\noindent\textbf{Step 1:} We first fix $\bq^1, \bq^2$ and can write:
\small
\begin{equation*}
    \Delta \dot{\bar{x}}_t  = \big(-\frac{2b_\alpha^2r_t}{c_\alpha} +b_x\big) \Delta\bar{x}_t - \frac{b_\alpha^2}{c_\alpha} \Delta q_t, \quad \Delta \bar{x}_0=0.
\end{equation*}
\normalsize 


\normalsize Then we have:
\small
\begin{equation*}
\begin{aligned}
    | \Delta \bar{x}_t| &= \int_0^t  \Big|  \Big(\big(-\frac{2b_\alpha^2r_s}{c_\alpha} +b_x\big) \Delta\bar{x}_s - \frac{b_\alpha^2}{c_\alpha} \Delta q_s\Big)\Big|ds\\
    &\leq \int_0^t  \Big| \big(-\frac{2b_\alpha^2r_s}{c_\alpha} +b_x\big) \Big| |\Delta\bar{x}_s| + \frac{b_\alpha^2}{c_\alpha}  |\Delta q_s| ds
\end{aligned}
\end{equation*}
\normalsize
By using Grönwall's inequality we have:
\small
\begin{equation*}
\begin{aligned}
    | \Delta \bar{x}_t|&\leq \int_0^t \frac{b_\alpha^2}{c_\alpha} |\Delta q_s|ds \times \exp\left(\int_0^t  \Big| -\frac{2b_\alpha^2r_s}{c_\alpha} +b_x \Big| ds \right)\\
    &\leq C_{\bar{x}} \int_0^T|\Delta q_s|ds
\end{aligned}
\end{equation*}
\normalsize
where the constant $ C_{\bar{x}} := \frac{b_\alpha^2}{c_\alpha}\exp\big(T \big(\frac{2b_\alpha^2 |r|_T}{c_\alpha} +|b_x| \big)\big) $ with $|r|_T =\sup_{s\in[0,T]} |r_s|<c_xe^{2RT}/4R$ only depends on $b_\alpha, b_x, c_\alpha, c_x$ and $T$.

\noindent\textbf{Step 2:} Next, given $\bxbar^1, \bxbar^2$, we have:
\small
\begin{equation*}
    -\Delta \dot{q}^\prime_t  = \Big(-\frac{2b_\alpha^2}{c_\alpha} r_t +b_x \Big)\Delta q^\prime_t +c_\mu (1-\lambda) \Delta\bar{x}_t,\quad \Delta q^{\prime}_T=0
\end{equation*}
\normalsize

By plugging in the bound we found in \textbf{Step 1} and by using Grönwall's inequality, we have: 
\small
\begin{equation*}
\begin{aligned}
|\Delta q\revision{^{\prime}}_t| &= \int_t^T \Big|\big(\frac{2b_\alpha^2}{c_\alpha} r_s -b_x \big)\Delta q^\prime_s -c_\mu (1-\lambda) \Delta\bar{x}_s \Big|ds\\
&\leq \int_t^T  \Big|\big(\frac{2b_\alpha^2}{c_\alpha} r_s +|b_x| \big)\Big | |\Delta q^\prime_s| + c_\mu (1-\lambda)  |\Delta\bar{x}_s|ds\\
&\leq \int_t^T  \Big|\big(\frac{2b_\alpha^2}{c_\alpha} r_s +|b_x| \big)\Big | |\Delta q^\prime_s|ds + T c_\mu (1-\lambda) C_{\bar x} \int_0^T |\Delta q_s| ds \\
&\leq \exp\Big(T\big(\tfrac{2b_\alpha^2}{c_\alpha} |r|_T +|b_x| \big)\Big) T c_\mu (1-\lambda) C_{\bar x} \int_0^T |\Delta q_s| ds.
\end{aligned}
\end{equation*}
\normalsize
Then we have $||\Delta q^{\prime}||_T \leq C_{q} ||\Delta q||_T$ where 
$C_q:= \exp\Big(T\big(\tfrac{2b_\alpha^2}{c_\alpha} |r|_T +|b_x| \big)\Big) T^2 c_\mu (1-\lambda) C_{\bar x} $.
When $T$ is small enough, we have $C_q<1$ hence leads to the mapping $q\mapsto q^\prime$ being a contraction mapping. Then by Banach fixed point theorem, there exists a unique solution to the coupled $\bq$ and $\bxbar$ system given $\boldsymbol{r}$.
Finally, $s_t$ process is determined uniquely given the unique $(q_t, r_t, \bar{x}_t)$ processes.
\end{proof}

The contraction argument above establishes existence and uniqueness for a sufficiently small time horizon. Extensions to arbitrary finite horizons may be possible under additional structural conditions; related continuation methods for forward-backward systems can be found in~\cite[Section 4.1.2]{CarmonaDelarue_book_I}. Alternatively, existence alone may be approached through Schauder’s fixed-point theorem, provided that the required compactness and invariant-set conditions are verified.

\subsection{Inverse Learning for the Level of Altruism and the Labor Cost}\label{subsec:inverse-learning}
Next, we take the regulator's viewpoint and study the inference of altruism level via altruism parameter \(\lambda\) and the labor cost parameter \(c_\alpha\). The method is inspired by an inverse learning approach, where the inference procedure is given by the following identifiability result. 
\begin{theorem}\label{thm:inference}
    Suppose the regulator observes a linear equilibrium strategy \(a(t,x)=k(t)x+b(t)\) for \(t\in[0,T]\), where \(k,b:[0,T]\to\R\) satisfy
    \begin{enumerate}
        \item there exists \(\rho>b_x\) such that for any \(t\in[0,T)\),
        \small\[k(t)=\frac{b_x^2-\rho^2}{b_\alpha}\frac{e^{2\rho(T-t)}-1}{(\rho-b_x)e^{2\rho(T-t)}+b_x+\rho}<0;\]\normalsize
        \item there exists \( c_0\leq0\) such that for any \(t\in[0,T]\),
        \small
        \[\begin{cases}
        b(t)=b_\alpha c_\mu c_0\int_t^T\tilde x(s)\exp\left\{b_x(s-t)+b_\alpha\int_t^sk(u)du\right\}ds,\\[1mm]
        \tilde x(t)=\bar\mu_0\exp\left\{b_xt+b_\alpha\int_0^tk(s)ds\right\}\\[1mm]
        \hspace{50pt}+b_\alpha\int_0^tb(s)\exp\left\{b_x(t-s)+b_\alpha\int_s^tk(u)du\right\}ds.\end{cases}\]
        \normalsize
    \end{enumerate}
    If the representative agent \(i\) reveals her labor cost parameter \(c_\alpha>0\), then the level of altruism \(1-\lambda\) can be uniquely identified by using
    \small
    \begin{equation}
        \label{eq:lambda-1}
        \widehat{\lambda}=1+\frac{c_\alpha b(t)\exp\left\{b_xt+b_\alpha\int_0^tk(s)ds\right\}}{b_\alpha c_\mu\int_t^T\tilde x(s)\exp\left\{b_xs+b_\alpha\int_0^sk(u)du\right\}ds},
    \end{equation}
    \normalsize
    for any given \(t\in[0,T)\) such that the denominator in \eqref{eq:lambda-1} is nonzero. Otherwise, assume \(b_x,b_\alpha>0\), 
    then \((c_\alpha,\lambda)\) can be uniquely identified  as
    \small
    \begin{align}
        \widehat{c}_\alpha&=\frac{b_\alpha^2c_x}{\rho_0^2-b_x^2},\label{eq:c-alpha}\\
        \widehat{\lambda}&=1+\frac{\widehat{c}_\alpha b(t)\exp\left\{b_xt+b_\alpha\int_0^tk(s)ds\right\}}{b_\alpha c_\mu\int_t^T\tilde x(s)\exp\left\{b_xs+b_\alpha\int_0^sk(u)du\right\}ds},\label{eq:lambda-2}
    \end{align}
    \normalsize
    where \(\rho_0\) is the unique root on \((b_x,+\infty)\) for
    \small
    \begin{equation}
        \label{eq:implicit}
        f(\rho)=k(t_0)+\frac{\rho^2-b_x^2}{b_\alpha}\cdot\frac{e^{2\rho(T-t_0)}-1}{(\rho-b_x)e^{2\rho(T-t_0)}+\rho+b_x}=0,
    \end{equation}
    \normalsize
    for any fixed \(t_0\in[0,T)\).
\end{theorem}
\begin{proof}
    Under the observed linear policy \(a(t,x)=k(t)x+b(t)\), the state dynamics \eqref{eq:dyn} implies that 
    \small
    \[\EE [X^a_t]=\bar\mu_0+\int_0^t[b_x+b_\alpha k(s)]\EE [X^a_s]+b_\alpha b(s)ds,\quad t\in[0,T].\]
    \normalsize
    Then \(\EE[X^a_t]=\tilde x(t)\) for all \(t\in[0,T]\). Also, note that conditions 1) and 2) are necessary for strategy \(a\) to be MI-MFNE specified in Proposition~\ref{prop:mi_FBODE}. 
    
    Suppose \(c_\alpha>0\) is given. Then from \eqref{FBODE:q}, we know that
    \small
    \[b'(t)+[b_x+b_\alpha k(t)]b(t)-\frac{b_\alpha c_\mu}{c_\alpha}(1-\lambda)\bar x(t)=0,\quad t\in[0,T).\]
    \normalsize
    Condition 2) implies that \eqref{eq:lambda-1} is time-invariant thus the conclusion follows.

    Suppose \(c_\alpha>0\) is not known. We can write \[f(\rho)=k(t_0)+f_1(\rho)\cdot f_2(\rho),\] with
    \small
    \[f_1(\rho)=\frac{\rho+b_x}{b_\alpha},\ \ f_2(\rho;t_0)=1-\frac{2\rho}{(\rho-b_x)e^{2\rho(T-t_0)}+\rho+b_x}.\]
    \normalsize
    Notice that $f(b_x) = k(t_0)<0$, and $f(\rho)\to +\infty$ as $\rho\to  +\infty$.
    Then on \((b_x,+\infty)\) and for every \(t_0\in[0,T)\), \(f_1\) is strictly increasing, and
    \small
    \[f_2'(\rho)=2\frac{2\rho(T-t_0)e^{2\rho(T-t_0)}(\rho-b_x)+b_x(e^{2\rho(T-t_0)}-1)}{[(\rho-b_x)e^{2\rho(T-t_0)}+\rho+b_x]^2}>0.\]
    \normalsize
    Hence, \(f\) is strictly increasing on \((b_x,+\infty)\) for every $0\leq t_0<T$, and with condition 1), \(f=0\) admits a unique root \(\rho_0\). Following \eqref{eq:riccati-sol}, \eqref{eq:c-alpha} uniquely holds; \eqref{eq:lambda-2} together with its uniqueness follows from same arguments for \eqref{eq:lambda-1} with \(c_\alpha=\widehat{c}_\alpha\).
\end{proof}
\begin{remark}
\begin{enumerate}
    \item Conditions 1) and 2) in Theorem~\ref{thm:inference} ensure that the strategy in the linear feedback form is indeed an equilibrium strategy in the MI-MFG specified in Section~\ref{sec:model}. Any violation will lead to the issue of model misspecification. 
    \item Theorem~\ref{thm:inference} tells us that the unknown parameters \((c_\alpha,\lambda)\) can be uniquely determined if the equilibrium policy is given. That is, our inverse learning problem does not suffer from nonidentifiability issue which is common in inverse reinforcement learning; see for instance \cite{cao2021identifiability}. This is likely due to the structural assumptions we have made to the game setting. The identifiability guarantee for general inverse learning problems for mean-field game/control remains to be explored in detail.

\end{enumerate}
\end{remark}


\section{Learning Algorithm and Results}
\label{sec:numerics}

Theorem~\ref{thm:inference} characterizes the inverse learning result for the altruism level and the labor cost parameter if the linear feedback form of the MI-MFNE is explicitly known by the regulator. In practice, however, the equilibrium strategy is only available as sample paths of actions. Moreover, the regulator is often only able to observe agents' state and action trajectories in the discrete time. Lastly, the regulator can collect trajectories for only finite number of agents, though this number can be large. To address the gap between the theoretical and the practical settings, we now propose the corresponding learning approach.

\subsection{Algorithm Setting}\label{subsec:num-setting}
With pre-specified \((c_x,c_\mu)\in(0,+\infty)^2\), the regulator is observing an \(N\)-agent mixed-individual game in Nash equilibrium over the finite time horizon \([0,T]\), \(T>0\). The regulator records agents' states and actions every \(\Delta t>0\) time. Therefore, at each round of observation, the regulator collects a set of \(N\) trajectories \(\pmb\tau=\{\tau^{(i)}\}_{i=1}^N\), where \(\tau^{(i)}=\left\{(X_t^{(i)},a^{(i)}_t)\right\}_{t\in\mathbb T}\) with \(\mathbb T=\left\{T\wedge i\Delta t:i=0,\dots,N_T, \Delta t = \frac{T}{N_T}
\right\}\). From \(\pmb\tau\), the regulator can estimate the population mean process \(\tilde X=\{\tilde X_t\}_{t\in\mathbb T}\) where for any \(t\in\mathbb T\),
\small\[\tilde X_t=\frac{1}{N}\sum_{i=1}^NX^{(i)}_t.\]\normalsize

We assume that the trajectories follow a time-homogeneous Gaussian transition kernel, 
\small
\begin{equation}\label{eq:kernel}
X^{(i)}_{T\wedge(t+\Delta t)}|(X^{(i)}_t,a^{(i)}_t)\sim N\left((1+b_x\Delta t )X^{(i)}_t+b_\alpha{a}^{(i)}_t \Delta t,\sigma^2(\Delta t\wedge(T-t))\right),
\end{equation}
\normalsize
where parameters \(b_x,b_\alpha>0\) are explicitly known by the regulator; initial distribution for \(X^{(i)}_0\) as well as the volatility parameter \(\sigma\) remain unknown. 

\subsection{Inverse Learning Procedure}
\paragraph{Learning the linear feedback form of equilibrium strategy} Since the trajectories \(\pmb\tau\) are actually collected from the MI-MFNE, then there exists a true underlying linear form \([0,T]\times\R\xmapsto{a}k(t)x+b(t)\) such that for any \(i=1,\dots, N\) and any \(t\in\mathbb T\),
\small\[a^{(i)}_t=k(t)X^{(i)}_t+b(t).\]\normalsize
Moreover, conditions 1) and 2) in Theorem~\ref{thm:inference} are satisfied. Therefore, we directly learn \((k,b)\) from linear regression,
\small
\begin{equation}
    \label{eq:reg}
    \begin{aligned}
    &(k,b)={\rm Regression}(\pmb\tau)\\
    &\hspace{10pt}=\argmin_{(\kappa,\beta)=\{(\kappa_t,\beta_t)\}_{t\in\mathbb T}}\frac{1}{N}\sum_{i=1}^N\sum_{t\in\mathbb T}\left(\kappa_tX^{(i)}_t+\beta_t-a^{(i)}_t\right)^2.
    \end{aligned}
\end{equation}
\normalsize

\paragraph{Learning the altruism level $1-\lambda\in[0,1]$ when \(c_\alpha>0\) is given}
From Proposition~\ref{prop:mi_FBODE} and Theorem~\ref{thm:inference}, we know that 
\small
\[k_t=-\frac{2b_\alpha}{c_\alpha}r_t,\ \ b_t=-\frac{b_\alpha}{c_\alpha}q_t,\ \ t\in\mathbb T,\]
\normalsize
where the paths \(r\) and \(q\) satisfy \eqref{FBODE:riccati} and \eqref{FBODE:q}, respectively. When \(c_\alpha>0\) is given, we construct inference result of \(\lambda\) according to \eqref{eq:lambda-1}. Notice that, by linearity of expectation, \(\tilde X\) is an unbiased estimator for \(\tilde x\) specified in condition 2) of Theorem~\ref{thm:inference}. We also introduce the following approximations of integrals in discrete time,
\small
\begin{align}
E_t&=\exp\left\{b_xt+b_\alpha\Delta t\sum_{u\in\mathbb T,u<t}k_u\right\},\label{eq:approx-t-1}\\[1mm] \mathcal E_t&=\Delta t\sum_{u\in\mathbb T,u\in[t,T)}\tilde X_uE_u,\label{eq:approx-t-2}\end{align}
\normalsize
for any \(t\in\mathbb T\setminus\{T\}\). Fix any \(t\in\mathbb T\setminus\{T\}\), construct the inferred \(\lambda\) as follows,
\small
\begin{equation}
    \label{eq:lambda-3}
    \widehat\lambda=1+\frac{c_\alpha b_tE_t}{b_\alpha c_\mu\mathcal E_t};
\end{equation}
\normalsize
equivalently, the inferred altruism level \(1-\lambda\) is given by
\small
\[1-\widehat{\lambda}=-\frac{c_\alpha b_tE_t}{b_\alpha c_\mu\mathcal E_t}.\]\normalsize


\paragraph{Jointly learning the altruism level \(1-\lambda\in[0,1]\) and the labor cost parameter \(c_\alpha>0\)}
Now that \(c_\alpha>0\) is to be estimated as well, according to Theorem~\ref{thm:inference}, we construct the estimated results of \(c_\alpha\) and \(\lambda\) according to \eqref{eq:c-alpha} and \eqref{eq:lambda-2} sequentially. First, we fix any \(t_0\in\mathbb T\setminus\{0,T\}\). Following \eqref{eq:implicit}, define 
\small
\[f(\rho;t_0)=k_{t_0}+\frac{\rho^2-b_x^2}{b_\alpha}\cdot\frac{e^{2\rho(T-t_0)}-1}{(\rho-b_x)e^{2\rho(T-t_0)}+\rho+b_x}.\]
\normalsize
Then, find the root 
\(\hat R\in(b_x,+\infty)\) such that \(f(\hat R;t_0)=0\). By \eqref{eq:c-alpha}, construct the learning result for \(c_\alpha\) as
\begin{equation}
    \label{eq:c-alpha-2}
    \widehat{c}_\alpha=\frac{b_\alpha^2c_x}{\hat R^2-b_x^2}.
\end{equation}
Subsequently, plugging \(\widehat{c}_\alpha\) into by \eqref{eq:lambda-3}, construct the learning result for \(\lambda\) as
\small
\begin{equation}
    \label{eq:lambda-4}
    \widehat{\lambda}=1+\frac{\widehat{c}_\alpha b_t E_t}{b_\alpha c_\mu\mathcal E_t},
\end{equation}
\normalsize
for any \(t\in[0,T)\); equivalently, the altruism level \(1-\lambda\) is given by
\small
\[1-\widehat{\lambda}=-\frac{\widehat{c}_\alpha b_t E_t}{b_\alpha c_\mu\mathcal E_t}.\]\normalsize

\paragraph{Algorithm}
We present the above learning procedure as in the following Algorithm~\ref{algo:SGD-SMFG}. To reduce the variance related to population mean estimate, we let the regulator conduct \(M\) rounds of observations over this \(N\)-agent system. The output is then the average over \(M\) observations (i.e., samples).

{\begin{algorithm}[!ht]
\caption{\small Estimation of altruism level and  labor cost}\label{algo:SGD-SMFG}
\begin{algorithmic}[1]
\small
\STATE \textbf{Setting:} number of total observations $M$ for $N$-agent system; time horizon \(T>0\), time steps $N_T\in\mathbb N^+$, discretization precision $\Delta t = \frac{T}{N_T}$, time index \(\mathbb T=[0,\Delta t,2\Delta t,\dots,N_T\Delta t]\), time point used in learning \(t_0\in\mathbb T\setminus\{0,T\}\); boolean variable \(I=\mathds{1}\{c_\alpha\text{ is given}\}\)
\vskip 1mm
\STATE \textbf{Initialization:}  null lists \(C_\alpha\) and \(\Lambda\); \(m=0\); 


\vskip1mm 
\WHILE{$m < M$}\vskip1mm
    \STATE{\(E=\text{all }0\text{ list of length }t_0\), \(\mathcal E=0\);}
    \STATE{Sample trajectories \(\pmb\tau\) for the \(N\) agents under MI-MFNE}\vskip1mm
    \STATE{\((k,b)={\rm Regression}(\pmb\tau)\)}\vskip1mm
    \STATE{\(\tilde X=\frac{1}{N}\sum_{i=1}^NX^{(i)}\)}\vskip2mm
    \FOR{\(t\) is in \([t_0,\dots,(N_T-1)\Delta t]\)}\vskip1mm
        \STATE{\(E\).append\(\left(E_t\right)\) according to \eqref{eq:approx-t-1}}
    \ENDFOR\vskip1mm
    \STATE{\(\mathcal E=\mathcal E_{t_0}\) according to \eqref{eq:approx-t-2} with \(E_u=E[u]\) for \(u\in[t_0,\dots,(N_T-1)\Delta t]\)}
    \vskip2mm
    \IF{\(I=1\)}
        \STATE{\(C_\alpha\).append\(\left(c_\alpha\right)\), \(\Lambda\).append\(\left(1+\frac{c_\alpha b_{t_0}E[t_0]}{b_\alpha c_\mu \mathcal E}\right)\)}
    \ELSE
        \STATE{Solve \(f(\rho;t_0)=0\) for \(\hat R\) on \((b_x,+\infty)\)}
        \STATE{\(\widehat{c}_\alpha=\frac{b_\alpha^2c_x}{\hat R^2-b_x^2}\)}
        \STATE{\(C_\alpha\).append\(\left(\widehat{c}_\alpha\right)\), \(\Lambda\).append\(\left(1+\frac{\widehat{c}_\alpha b_{t_0}E[t_0]}{b_\alpha c_\mu \mathcal E}\right)\)}
    \ENDIF
    \STATE{$m++$};
\ENDWHILE
\RETURN \(\texttt{mean}(C_\alpha)\), \(\texttt{mean}(1-\Lambda)\), 
\(\texttt{std}(C_\alpha)\), 
\(\texttt{std}(\Lambda)\).
\end{algorithmic}
\end{algorithm}}

\subsection{Experimental Results}
We test Algorithm~\ref{algo:SGD-SMFG} with the following parameter settings:
\small
$$
T = 1, \ b_\alpha = 1,\ b_x = 0.5,\ c_\mu = 1,\ c_x = 1. 
$$
\normalsize
We generate our observations of $\{\alpha^{(i)}, X^{(i)}\}_{t\in\mathbb{T}}$ through~\eqref{eq:kernel}, in which the initial distribution $\xi$ of each agent follows an i.i.d normal distribution with mean $\bar\mu_0 = 0.5$ and variance the volatility $\sigma_0^2 = 0.01$ and where $\sigma = 0.1$. 
With the ground truth for the altruism level and labor costs being
\small
$$
1-\lambda = 0.1,\qquad c_\alpha = 1,  
$$
\normalsize
we obtain simulated $100$-agent system action-state pairs using \(N_T=1000\) time steps so that the discretization precision is $\Delta t = 0.001$. We compare learned altruism levels and labor costs from the proposed algorithm to the ground truth.

\paragraph{Inference for altruism level}
When the labor cost $c_\alpha=1$ is known, we choose $t_0 = T/2$ to apply Algorithm~\ref{algo:SGD-SMFG}. In a $100$-agent model, we can obtain the following estimation for the altruism level $1-\lambda$ as the number of observations $M$ of the system increases from 1 to $50$ as illustrated in Figure~\ref{fig:only altruism level}.
The estimated interval is constructed by
{\small
$$
[\texttt{mean}(1-\Lambda) - 2.58 \texttt{std}(\Lambda), \texttt{mean}(1-\Lambda) + 2.58 \texttt{std}(\Lambda)].
$$
}
Here, 2.58 corresponding to the 99.5\% upper quantile of standard normal distribution. 
\begin{figure}
    \centering
    \includegraphics[width=0.47\textwidth]{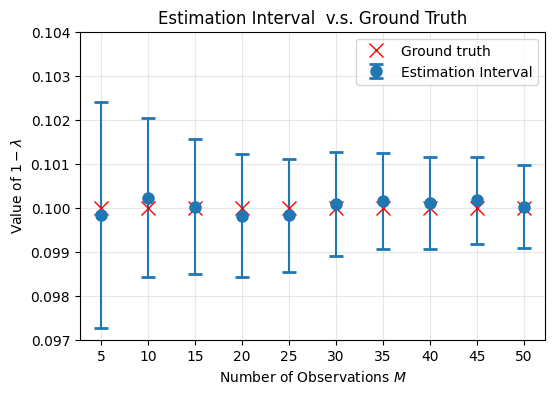}
    \caption{Inverse learning for altruism level $1-\lambda$ when $c_\alpha$ is known.}
    \label{fig:only altruism level}
\end{figure}
We can clearly see that, the overall estimation is very accurate, and the estimation interval becomes smaller  when the number of observations increases, as expected.

\paragraph{Inference for labor cost and altruism level}
When the labor cost $c_\alpha$ is unknown, we fix $t_0=T/2$ to apply Algorithm~\ref{algo:SGD-SMFG}. In a $100$-agent model,  the estimation error for $c_\alpha$ via Algorithm~\ref{algo:SGD-SMFG}  is $|c_\alpha - \widehat{c}_\alpha| \approx 2.2\times 10^{-10}$, with a standard deviation $O(10^{-16})$. Thus we can claim the success of recovering the labor costs. Figure~\ref{fig:altruism level} shows the estimation of the altruism level $1-\lambda$ with increasing number of observations, where the estimation interval is constructed in the same manner as before. As the number of observations $M$ of the system increases from 1 to $50$, the estimation gets more and more accurate as in the previous case, with only slightly larger variance (where the difference of variance is of order $10^{-4}$). 

\begin{figure}
    \centering
    \includegraphics[width=0.47\textwidth]{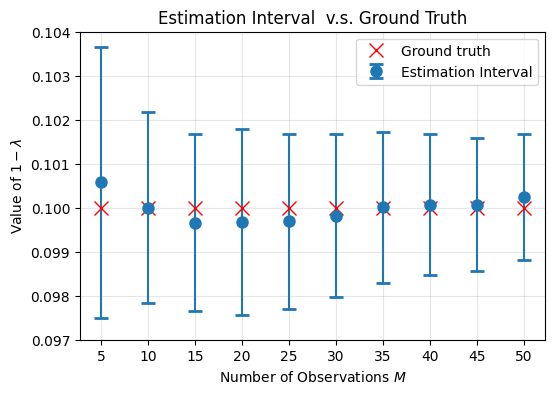}
    \caption{Inverse learning for  altruism level $1-\lambda$ when $c_\alpha$ is unknown.}
    \label{fig:altruism level}
\end{figure}

\section{Conclusion and Future Work}
\label{sec:conclusion}

In this paper, we study inverse learning in MI-MFGs, with the goal of identifying both labor cost parameters and the altruism levels in a data-driven way from observed equilibrium behavior. We introduced and proved forward-backward PDE-based characterization for the MI-MFNE and established its existence and uniqueness. Furthermore, we analyzed the inverse learning approach and established identifiability results. Our results constitute the first inverse learning framework for MI-MFGs and offer theoretical guarantees for altruism and cost parameter recovery. We further provided learning approach and our experiments illustrate the effectiveness of the proposed approach. Our findings contribute to the understanding of the agent behavior in large systems in which both cooperative and non-cooperative tendencies can occur.

Several directions remain open for future research. In the present model, the relationship between the cost structure and the equilibrium strategy is established through a variational approach. A natural next step is to develop an MI-MFG formulation that admits a dynamic programming principle, thereby enabling inverse \emph{reinforcement} learning methods for recovering altruism levels from more general forms of observed behavior. Such a formulation would require equilibrium feedback strategies to be defined on an augmented state-and-distribution space, reflecting the dependence of the agents' decisions on both their individual states and the population distribution. Another important direction is to formulate a genuine Stackelberg game between a regulator and a population of mixed-individual agents. This would allow the regulator not only to infer latent preferences, but also to use the inferred model to design \textit{optimal} incentives and anticipate the resulting equilibrium response.

We also plan to extend the proposed learning methodology to settings in which the regulator has access only to state trajectories, only to control trajectories, or to aggregate population-level observations. Further extensions may incorporate multidimensional state dynamics, nonlinear cost structures, heterogeneous populations, and imperfect or partially observed equilibrium behavior. In these more general settings, establishing identifiability will remain a fundamental challenge, since different combinations of preferences and structural parameters may generate observationally similar behavior. Developing finite-sample guarantees, uncertainty quantification, and computationally efficient learning algorithms will therefore be essential for translating inverse MI-MFG methods into reliable tools for empirical behavioral modeling and data-driven policy design.

\bibliographystyle{plain}
\bibliography{ref}

\end{document}